\documentclass[12pt]{article}
\usepackage{amssymb}
\usepackage{amsmath}
\usepackage{amsthm}
\usepackage{authblk}
\usepackage{color}
\usepackage{comment}
\usepackage{geometry}		
\usepackage{graphicx}

\DeclareFontFamily{OT1}{pzc}{}
\DeclareFontShape{OT1}{pzc}{m}{it}{<-> s * [1.10] pzcmi7t}{}
\DeclareMathAlphabet{\mathpzc}{OT1}{pzc}{m}{it}

\let\originalleft\left
\let\originalright\right
\renewcommand{\left}{\mathopen{}\mathclose\bgroup\originalleft}
\renewcommand{\right}{\aftergroup\egroup\originalright}

\begin{document}

\newcommand{\ee}{\varepsilon}

\newcommand{\removableFootnote}[1]{\footnote{#1}}

\newtheorem{theorem}{Theorem}[section]
\newtheorem{corollary}[theorem]{Corollary}
\newtheorem{lemma}[theorem]{Lemma}
\newtheorem{proposition}[theorem]{Proposition}

\theoremstyle{definition}
\newtheorem{definition}{Definition}[section]
\newtheorem{example}[definition]{Example}

\theoremstyle{remark}
\newtheorem{remark}{Remark}[section]

\title{
Normal forms, differentiable conjugacies and elementary bifurcations of maps.
}
\author[$\dagger$]{P.A.~Glendinning}
\author[$\ddagger$]{D.J.W.~Simpson}
\affil[$\dagger$]{Department of Mathematics, University of Manchester, Manchester, UK}
\affil[$\ddagger$]{School of Mathematical and Computational Sciences, Massey University, Palmerston North, New Zealand}

\maketitle



\begin{abstract}
We strengthen the standard bifurcation theorems for saddle-node, transcritical, pitchfork, and period-doubling bifurcations of maps. Our new formulation involves adding one or two extra terms to the standard truncated normal forms with coefficients determined by algebraic equations. These extended normal forms are differentiably conjugate to the original maps on basins of attraction and repulsion of fixed points or periodic orbits. This reflects common assumptions about the additional information in normal forms despite standard bifurcation theorems being formulated only in terms of topological equivalence. 
\end{abstract}

\section{Introduction}
\label{sect:intro}
\setcounter{equation}{0}

In most textbooks bifurcation theorems contain two parts: a \textbf{skeleton} in which the existence of particular solutions (e.g.~fixed points or periodic orbits) is established as a function of parameters, and a local \textbf{equivalence} in which the dynamics away from the skeleton is described. The local equivalence is topological, but in this paper we show that it can be made differentiably conjugate to simple polynomial normal forms.

The skeleton is usually established using the Implicit Function Theorem \cite{Devaney, Iooss, Kuznetsov, Wiggins}, although versal deformations of singularities can also provide this information \cite{Montaldi}, as can local asymptotic expansions \cite{Glendinning}. The local equivalence often relates the dynamics to a `typical' simple example. This can either be by a rigorous change of coordinates to a simple normal form \cite{Takens1973}, or to a truncated version of the normal form \cite{G&H, Kuznetsov} for which it is claimed that the local dynamics is topologically equivalent to the general system being studied.

For the four simplest local bifurcations of one-dimensional maps, the truncated normal forms used to describe these bifurcations are listed in the second column of Table~\ref{table:1}. This table also lists our modifications to the normal forms introduced below. For multi-dimensional maps the normal forms can be obtained by first reducing to a one-dimensional centre manifold.

\begin{table}
\centering
\begin{tabular}{|| c ||l l ||}
\hline\hline
\textrm{Bifurcation} & \textrm{Standard normal form} & \textrm{Additional terms}\\
\hline\hline ~&~&~\\[-3.7mm]									
\textrm{saddle-node} & $x_{n+1}=x_n+\nu -x_n^2$ & $+\,ax_n^3$\\[.3mm]    
\hline ~&~&~\\[-3.7mm]
\textrm{transcritical} & $x_{n+1}=x_n+x_n(\nu -x_n)$ & $+\,ax_n^3$\\[.3mm]
\hline ~&~&~\\[-3.7mm]
\textrm{pitchfork} & $x_{n+1}=x_n+x_n(\nu -x^2_n)$ & $+\,a x_n^5 + b \nu x_n^2$\\[.3mm]
\hline ~&~&~\\[-3.7mm]
\textrm{period-doubling} & $x_{n+1}=-x_n - x_n(\nu - x_n^2)$ \hspace{2mm}   
& $+\,a x_n^5$\\[.3mm]
\hline\hline
\end{tabular}
\caption{
Standard normal forms and additional terms that make it possible to create differentiable conjugacies.
}
\label{table:1}
\end{table}

The relationship between the two elements of the analysis, change of coordinates and topological equivalence, is often unclear.
If only topological equivalence is required, then the continuous parameter $\nu$ of the truncated normal form of the
saddle-node bifurcation, for example, could be replaced by the sign of $\nu$, a discrete parameter (see Theorem~\ref{thm:mono}).
Since a topological conjugacy does not depend on the existence of derivatives, a family of piecewise-affine maps could also be used.

Despite this, many textbooks take great care to show that the truncated normal forms of Table~\ref{table:1} can be obtained by smooth changes of coordinates to leading order, e.g.~\cite{Kuznetsov}. Clearly there is an unwritten assumption that the truncated normal forms carry more information than the topological
conjugacies stated in the theorems, and indeed the skeleton arguments carry information about the parameter dependence of some solutions without recourse to the normal form. But in that case why introduce the normal forms and changes of coordinates as anything other than simple examples?

In this paper we will show that the normal forms, modified as in Table~\ref{table:1}, do indeed carry more information. More explicitly, there are local diffeomorphisms between parts of the dynamics of the general map and the corresponding parts of the dynamics of the extended normal form. We also provide equations connecting the new parameterisation to the parameter dependent coefficients of the new terms. Our proof of local differentiable equivalence does not involve actually creating the coordinate changes, so none of the technical issues regarding the convergence of infinitely many successive coordinate changes will be needed. 

The extra information used in our analysis is the multiplier (or stability coefficient) of a periodic orbit.
The multiplier of a fixed point is the derivative of the map evaluated at the fixed point.
For a periodic orbit the multiplier is the product of the derivative evaluated at the points on the orbit.
Our method is to equate the multipliers of the fixed points of the normal form to those of the original map
and the additional terms in Table~\ref{table:1} allow us to do this.
Then we use local linearization theorems \cite{Belitskii1986,Ofarrell2009,Ofarrell2011,Sternberg} to show 
the existence of differentiable conjugacies on the basins of attraction and
basins of repulsion of the fixed points and periodic orbits of the skeletons.
This extends the standard two step analysis (skeleton and topological equivalence) to the following four steps
(once the reduction to a one-dimensional centre manifold has been made):
\begin{itemize}
\item a \textbf{skeleton} in which the existence of particular solutions is established as a function of parameters;
\item a \textbf{normal form skeleton} in which the existence of particular solutions
is established as a function of parameters of the normal form;
\item \textbf{multiplier equivalence} in which it is shown that the multipliers of the skeleton are equal
to the multipliers of corresponding solutions of the normal form skeleton for appropriate values of coefficients; and
\item \textbf{local differentiable conjugacy} in which the dynamics away from the skeleton is described
via one or more differentiable conjugacies.
\end{itemize}

In the remainder of this paper we go through the standard bifurcation theorems in this stronger format.
In \S\ref{sect:basics} we give the definitions and sketch the conjugacy theorems needed to develop this approach.
In \S\ref{sect:results} we give the strengthened bifurcation theorems in full;
the subsequent four sections go through the four fundamental cases. 
This starts with the transcritical bifurcation (\S\ref{sect:transcritical})
as it is in many ways the easiest to handle, and we do this without full rigour to show how the method works.
As in most expositions we do not specify neighbourhoods as we go through the argument.
A more detailed account is given in \S\ref{sect:saddlenode} where we treat the saddle-node bifurcation.
Paradoxically it is the case with no fixed points, the easiest to work with on the whole real line,
which presents the greatest challenge when restricted to a fixed interval in space with varying parameter.
We then sketch the equivalent approach for the pitchfork bifurcation and period-doubling bifurcation in
\S\ref{sect:pitchfork} and \S\ref{sect:periodDoubling}.
We conclude in \S\ref{sect:conc}
with a short summary of how our new approach fits in with other bifurcations.
For each bifurcation there is a tension between having the strongest dynamic equivalence
given the types of dynamics that are generated and keeping the analysis as simple as possible.
We argue that the approach here via differentiable conjugacy is optimal for the elementary bifurcations,
but that it is not so appropriate in more complicated bifurcations.

\section{Topological versus differentiable conjugacies}
\label{sect:basics}
\setcounter{equation}{0}

Near local bifurcations maps are monotonic on a neighbourhood of the bifurcation point, increasing 
for the saddle-node, transcritical and pitchfork bifurcations and decreasing for the period-doubling bidurcation.
In this section we bring together the technical results needed to prove the new bifurcation theorems stated in \S\ref{sect:results}.


\begin{definition}
Maps $f:U\to \mathbb{R}$ and $g:V\to \mathbb{R}$ are {\em topologically conjugate}
if there exists a homeomorphism (continuous bijection with continuous inverse) $h:U\to V$
such that $h \circ f=g\circ h$.
If $h$ is a diffeomorphism then $f$ and $g$ are {\em differentiably conjugate}.
If $h$ is of class $C^r$, $r \ge 1$, then $f$ and $g$ are $C^r$-\emph{conjugate}.
\label{df:conjugate}
\end{definition}

\medskip
If $x$ is a fixed point of $f$ then $y=h(x)$ is a fixed point of $g$
and referred to as the corresponding fixed point of $g$.
If $h$ is differentiable then $x$ and $y$ have the same multiplier: $f^\prime (x)=g^\prime (y)$.
The conjugacy $h$ is essentially a coordinate transformation, so if $y_n=h(x_n)$ then 
\[
y_{n+1}=h(x_{n+1})=h(f(x_n))=h(f(h^{-1}(y_n))), \qquad \textrm{i.e.}~ g=h\circ f \circ h^{-1}.
\]

The following result shows just how weak the condition of topological conjugacy is for increasing maps of the real line. We are not sure to whom this result should be attributed, but it is part of the folklore of the subject.

\medskip
\begin{theorem}
Suppose $f:\mathbb{R}\to \mathbb{R}$ and $g:\mathbb{R}\to \mathbb{R}$ are
increasing homeomorphisms with precisely $n\ge 0$ fixed points, $x_k$ and $y_k$ respectively for $k=1,\dots ,n$, with
$x_0=-\infty <x_1 <\dots < x_n<\infty =x_{n+1}$ and  $y_0=-\infty <y_1 <\dots < y_n<\infty =y_{n+1}$.
Then $f$ and $g$ are topologically conjugate by an increasing
homeomorphism if and only if
the sign of $f(x)-x$ on $(x_k,x_{k+1})$ and the sign of $g(y)-y$ on $(y_k,y_{k+1})$ is equal for each $k=0,\dots ,n$.
\label{thm:mono}
\end{theorem} 

Here we prove Theorem \ref{thm:mono} by constructing $h$ on each $(x_k,x_{k+1})$ so that it maps an orbit of $f$
on this interval to an orbit of $g$ on $(y_k,y_{k+1})$, as in Fig.~\ref{fig:a}.

\begin{figure}[t!]
\begin{center}
\includegraphics[height=5.5cm]{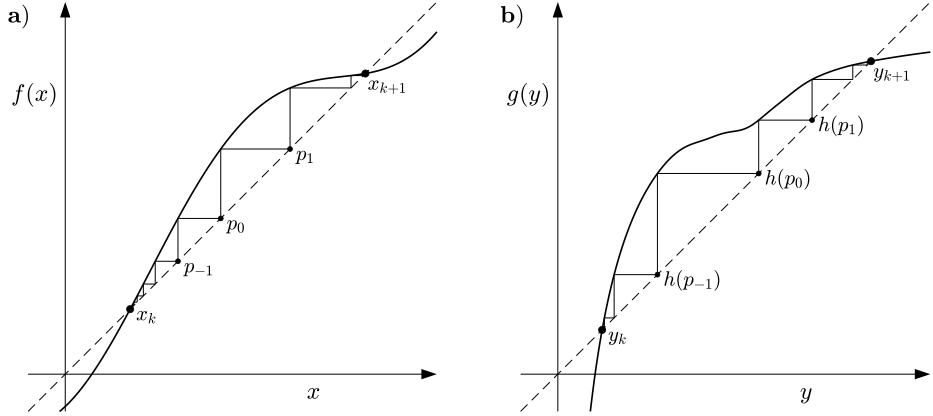}
\caption{
A sketch of the increasing maps $f$ and $g$ of Theorem \ref{thm:mono}.
As indicated $h$ is constructed so that it maps an orbit of $f$ to an orbit of $g$.
\label{fig:a}
}
\end{center}
\end{figure}

\medskip
\begin{proof}
First suppose $h \circ f = g \circ h$ for an increasing homeomorphism $h$.
Then each fixed point of $f$ maps to some fixed point of $g$.
Further $h(x_k) = y_k$ for each $k$ because $h$ is increasing.
Thus $h$ maps $(x_k,x_{k+1})$ to $(y_k,y_{k+1})$ for each $k = 0,\dots,n$.
Now choose any $k = 0,\dots,n$ and $x \in (x_k,x_{k+1})$.
The sign of $f(x) - x$ is the same as the sign of $h(f(x)) - h(x)$ because $h$ is increasing.
By $h \circ f = g \circ h$ this is the sign of $g(y) - y$ where $y = h(x) \in (y_k,y_{k+1})$.
Thus $f(x) - x$ and $g(y) - y$ have the same sign as required.

Conversely suppose $f(x) - x$ on $(x_k,x_{k+1})$ has the same sign as $g(y) - y$ on $(y_k,y_{k+1})$ for each $k$. 
Define $h$ on the fixed points $\{ x_k \}_1^n$
by $h(x_k) = y_k$.
For each $k \in \{0,\dots ,n\}$ choose some $p_0 \in (x_k,x_{k+1})$ and $q_0 \in (y_k,y_{k+1})$ and let
\[
p_i = f^i(p_0), \qquad q_i = g^i(q_0),
\]
for all $i \in \mathbb{Z}$.
Suppose $f(x) - x > 0$ on $(x_k,x_{k+1})$.
Then $\{ p_i \}$ and $\{ q_i \}$
are strictly increasing sequences and 
\[
\lim_{i \to -\infty} p_i = x_k, \qquad
\lim_{i \to \infty} p_i = x_{k+1}, \qquad
\lim_{i \to -\infty} q_i = y_k, \qquad
\lim_{i \to \infty} q_i = y_{k+1} \,,
\]
as in Fig.~\ref{fig:a}.
(If $f(x) - x$ and $g(y) - y$ are negative the sequences are decreasing and the limits are switched.)
Now define $h:[p_0,p_1] \to [q_0,q_1]$
using any continuous strictly increasing function, say 
\[
h(x)= q_0+\frac{q_1-q_0}{p_1-p_0}(x-p_0).
\]
Then define $h$ inductively for $x\in [p_i, p_{i+1}]$, $i\ge 1$ by
\[
h(x)=\left( g\circ h \circ f^{-1} \right)\!(x),
 \]
and for $x\in [p_{-(i+1)}, p_{-i}], i\ge 0$ by
\[
h(x)=\left( g^{-1}\circ h \circ f \right)\!(x).
 \]
By construction this is continuous on the images and preimages of $p_0$,
continuous and strictly increasing at all other points, and $g\circ h = h\circ f$.
\end{proof}

The remarks in \S\ref{sect:intro} about the inadequacies of topological conjugation stem from this theorem.
Stronger results of course need more conditions.
These are supplied by an extension of Sternberg's linearization result \cite{Sternberg} due to Belitskii \cite{Belitskii1986}.

\begin{theorem}(Belitskii) Suppose $f:\mathbb{R}\to \mathbb{R}$ and $g:\mathbb{R}\to \mathbb{R}$ are
strictly increasing $C^r$ diffeomorphisms, $r\ge 1$,
and both maps have precisely $n\ge 1$ fixed points, $x_k$ and $y_k$ respectively for $k=1,\dots ,n$, with $x_0=-\infty <x_1 <\dots < x_n<\infty =x_{n+1}$ and  $y_0=-\infty <y_1 <\dots < y_n<\infty =y_{n+1}$. Suppose in addition that 
\[
f^\prime (x_k)=g^\prime (y_k) \ne 1, \qquad \text{for all}~ k=1, \dots ,n.
\]
Then $f$ restricted to $(x_{k-1},x_{k+1})$ and $g$ restricted to $(y_{k-1},y_{k+1})$ are $C^{r-1}$-conjugate
for each $k=1, \dots ,n$.
\label{thm:difffp}
\end{theorem}

An outline of a proof of this result is given in Appendix \ref{app:Bel}.
If $f$ and $g$ have no fixed points then a minor variation of the construction in the proof of 
Theorem \ref{thm:mono} shows that there is a topological conjugacy which is smooth everywhere 
except at the `boundary' points $p_i$
(these are the points of an orbit of $f$).
Since there are no fixed points 
we cannot use Sternberg's result \cite{Sternberg} to generate intervals on which both $p_i$ and $f(p_i)$ or $f^{-1}(p_i)$ are contained in an open set on which they are both continuously differentiable.  It is possible to use a lemma due originally to Borel (see e.g.~\cite{Ofarrell2011}) based on formal power series and Taylor's theorem to prove that the conjugating function can be chosen to be continuously differentiable at these end points. Here we adopt a slightly different, and simpler, strategy.

\begin{theorem}
Suppose that $f:\mathbb{R}\to \mathbb{R}$ and $g:\mathbb{R}\to \mathbb{R}$ are increasing $C^r$
diffeomorphisms, $r\ge 1$, with no fixed points, and that both $f(x)-x$ and $g(x)-x$ have the same sign. 
Then there exists an increasing $C^{r}$-conjugacy $h:\mathbb{R}\to \mathbb{R}$ with $g\circ h = h\circ f$.
\label{thm:diffnofp}
\end{theorem}

\begin{proof}
Suppose $f(x)-x>0$ so the images and preimages $\{ f^i(x) \}_{-\infty}^\infty$ form an increasing sequence. Choose any point, for example $x=0$, and define the conjugating function from an open neighbourhood $U$ of $0$ to a neighbourhood $V$ of $0$, say by letting $h(x)=x$ on $U$.
By choosing $U$ smaller if necessary we may assume that the closures of the images $f^i(U)$ are disjoint
for all $i \in \mathbb{Z}$.
Now define $h$ on $f^i(U)$, $i \in\mathbb{Z}$, by
\[
h(x) = \left( g^i\circ h \circ f^{-i} \right)\!(x).
\]
This immediately implies $h=g\circ [g^{i-1}\circ h \circ f^{-(i-1)}]\circ f^{-1}$ and since the composition of maps in square brackets is simply $h$ on $f^{i-1}(U)$ the functions $h$ defined in this way satisfy $h\circ f=h\circ g$.

Now write $f^i(U)=(u_i,u^\prime_i)$. Using a $C^\infty$ jump function extend $h$ on $[ u^\prime_0,u_1]$ so that $h$ is a $C^r$ diffeomorphism on $(u_0, u^\prime_1)$. We can now push $h$ forwards and backwards as in
the proof of Theorem \ref{thm:mono},
and the $C^{r}$ differentiability of $h$ follows from the fact that $f$ and $g$ are $C^r$.  
\end{proof}

The non-hyperbolic case is more subtle and is treated by Takens \cite{Takens1973} in the $C^\infty$ case and Kuczema \emph{et al.} \cite{KCG1990} (Theorem 8.4.5) in the $C^2$ case, see also \cite{Young}.

\begin{theorem}(\cite{KCG1990,Takens1973})
Suppose that $f$ and $g$ are $C^2$ and 
\[
f(x)=x+ax^p+O(x^{p+1}), \quad g(x)=x+bx^p+O(x^{p+1}), \quad p= 2, 3, \dots .
\]
Then if $ab>0$ there is a $C^1$ increasing conjugacy between $f$ and $g$ on neighbourhoods of the origin, and 
in particular $f$ is differentiably conjugate to $y\to y+\textrm{sign}(a)y^p$. 
If $f$ is $C^\infty$ then there is a $C^\infty$ increasing conjugacy to
\[
y_{n+1}=y_n+\textrm{sign}(a) y_n^p+\alpha_py_n^{p+1},
\]
where, if $f^{(r)}(x)$ is the $r^{th}$ derivative of $f$,
$\alpha_p=\left(\frac{p!}{|f^{(p)}(0)|}\right)^{\frac{p}{p-1}}\frac{f^{(p+1)}(0)}{(p+1)!}.$
\label{thm:takens}\end{theorem}

Note that
\begin{equation}\label{eq:alphas}
\alpha_2=\frac{2f_{xxx}}{3f_{xx}^2} \Big|_{x=0}, \quad
\alpha_3=\sqrt{\frac{3}{8}}\frac{f_{xxxx}}{|f_{xxx}|^{\frac{3}{2}}} \Bigg|_{x=0},
\end{equation}
which are quantities that appear in our results below.

To the best of our knowledge it is not known whether the conjugacy can be made smoother under the assumption
that $f$ and $g$ are $C^r$, $3\le r<\infty$.
The statement of Kuczema \emph{et al.}~\cite{KCG1990} allows for non-integer values of $p$
but we have stated the result for the cases needed below.
However, their construction is for the half neighbourhood $[0,a)$ and to extend to $(-a,0]$ as a $C^1$ function the transformation $x\to -x$ can be used giving coefficents $(-1)^{p+1}a$ and $(-1)^{p+1}b$ which lead to the same derivative for the conjugating function at $0$ when transformed back to $x\le 0$.   

As mentioned in \S\ref{sect:intro}, the main tool for proving the existence and persistence of periodic orbits is the Implicit Function Theorem. This comes in many forms. For our purposes the statement below, from \cite{IFT} (Theorem 2.1), gives the important smoothness conditions we need. 

\begin{theorem}(Implicit Function Theorem)
Let $U$ be an open neighbourhood of $(x_0,y_0) \in \mathbb{R}^p\times \mathbb{R}^q$. Let $F : U \to \mathbb{R}^q$ be a $C^r$ function
where $1 \le r\le \infty$ and suppose that $F(x_0,y_0) = 0$. If $DF_y(x_0,y_0)$ is non-singular then there exists an open neighbourhood $V$ of $(x_0,y_0)$ with $V\subseteq U$, such that $DF_y(x, y)$  is non-singular for all $(x, y)\in V$ and an open neighbourhood of $x=x_0$, $W\subset \mathbb{R}^p$, and a unique function $f : W\to \mathbb{R}^q$ which is $C^r$ such that $f(x_0) = y_0$ and, for all $x \in W$,
\[
(x, f(x)) \in V \quad \textrm{and} \quad F(x, f(x)) = 0 .
\]
Moreover,  
\[
Df (x_0) = -[DF_y(x_0,y_0)]^{-1}DF_x(x_0,y_0).
\]
\label{thm:IFT}
\end{theorem}
If $p=q=1$ the condition that $DF_y(x_0,y_0)$ is non-singular, i.e.~$\det \left( DF_y(0,0) \right) \ne 0$, is equivalent to
$\frac{\partial F}{\partial y}(x_0, y_0)\ne 0$ in which case  $y=f(x)= f^\prime (x_0)(x-x_0)+o(|x-x_0|)$ locally, with 
\[
f^\prime (x_0)=-\left(\frac{\partial F}{\partial y}(x_0,y_0)\right)^{-1}\frac{\partial F}{\partial x}(x_0,y_0).
\]
To summarise, given $q$ $C^r$ real equations in $p+q$ variables, the Implicit Function Theorem provides conditions under which there is a locally unique $C^r$ function determining $q$ of the variables as a function of the remaining $p$ variables on which the original $q$ real equations are satisfied.

\section{Bifurcation Theorems}
\label{sect:results}
\setcounter{equation}{0}

The conditions of Belitskii's Theorem do not imply the existence of global conjugacies,
so the best we can hope for in general is local.
This is reflected in the theorems below.

Each theorem concerns a map $x_{n+1} = f(x_n,\mu)$ and an extended normal form.
For the saddle-node bifurcation this form is
\begin{equation}
y_{n+1} = g(y_n,\nu,a) = y_n + \nu - y_n^2 + a y_n^3 \,.
\label{eq:snfnf}
\end{equation}

\begin{theorem}(Saddle-node bifurcation)
Suppose $f$ is $C^r$, $r \ge 4$, and 
\begin{equation}
f(0,0) = 0, \quad f_x(0,0) = 1, \quad f_\mu(0,0) > 0, \quad \textrm{and} \quad f_{xx}(0,0) < 0.
\label{eq:snf}
\end{equation}
Let $g$ be the truncated normal form (\ref{eq:snfnf}). There exists a neighbourhood $N$ of $x=0$
and $\mu_0 > 0$ such that if $\mu \in (-\mu_0,0)$ then $f$ has no fixed points in $N$,
and if $\mu \in (0,\mu_0)$ then $f$ has two fixed points in $N$.
Moreover, there exists a neighbourhood $M$ of $y=0$
and continuous functions $\nu(\mu)$ and $a(\mu)$ 
with
\begin{equation}
\nu(0) = 0, \quad
\lim_{\mu \to 0^+} \nu'(\mu) = -\frac{f_\mu f_{xx}}{2} \Big|_{(0,0)}, \quad
a(0) = \frac{2 f_{xxx}}{3 f_{xx}^2} \Big|_{(0,0)},
\label{eq:snFormulae}
\end{equation}
such that with $\tilde{g}(y,\mu) = g(y,\nu(\mu),a(\mu))$,
\begin{enumerate}
\item
if $\mu \in (-\mu_0,0)$ then $f|_N$ is $C^r$-conjugate to $\tilde{g}|_M$,
\item
if $\mu = 0$ then $f|_N$ is $C^1$-conjugate to $\tilde g|_M$, and
\item
if $\mu \in (0,\mu_0)$ then $f|_N$ and $\tilde g|_M$ are
$C^{r-1}$-conjugate on the basins of attraction/repulsion of their corresponding fixed points.
\end{enumerate}
\label{thm:snb}
\end{theorem}

Note that the signs of $f_{\mu}(0,0)$ and $f_{xx}(0,0)$ (provided they are non-zero)
can be chosen as stated without loss of generality by using the orientation-reversing transformations
$x \mapsto -x$ and $\mu \mapsto -\mu$ where necessary.  The extra differentiablility ($C^r$, $r\ge 4$) is required 
for the application of the Implicit Function Theorem to obtain the parameters $\nu$ and $a$ of the normal
form as a function of $\mu$. The maps remain $C^r$ so the conjugacies in neighbourhoods of the 
fixed points are $C^{r-1}$ by Theorem~\ref{thm:difffp}. 

One can show that $\nu(\mu)$ is $C^1$ except possibly at $\mu = 0$.
The value of the derivative $\nu'(\mu)$ as $\mu$ approaches $0$ from the right is well-defined
and takes the value $-\frac{f_\mu f_{xx}}{2} \big|_{(0,0)}$ as indicated above.
This value represents the rate at which $\mu$ unfolds the bifurcation relative to $\nu$ in the extended normal form.
Furthermore, as evident in the proof below,
for $\mu \in (0,\mu_0)$ if we treat $\nu$ and $a$ as functions of $m = \sqrt{\mu}$,
then $\nu$ is $C^{r-2}$ and $a$ is $C^{r-3}$. The value of $a(0)$ is precisely Takens' $\alpha_2$ in (\ref{eq:alphas}).

For the remaining bifurcations we make the (standard) simplification of assuming the origin is
always a fixed point.
Specifically we assume
\begin{equation}
f(0,\mu) = 0, \quad \text{for all $\mu$ in a neighbourhood of $0$}.
\label{eq:originAlwaysFixed}
\end{equation}
For transcritical the extended normal form is
\begin{equation}
y_{n+1} = g(y_n,\nu,a) = y_n + \nu y_n - y_n^2 + a y_n^3 \,.
\label{eq:tcfnf}
\end{equation}

\begin{theorem}(Transcritical bifurcation)
Suppose $f$ is $C^r$, $r \ge 4$, satisfying \eqref{eq:originAlwaysFixed} and
\begin{equation}	
f_x(0,0) = 1, \quad f_{xx}(0,0) < 0, \quad \textrm{and} \quad f_{x \mu}(0,0) > 0.
\label{eq:tcf}
\end{equation}
Let $g$ be the truncated normal form (\ref{eq:tcfnf}). There exists a neighbourhood $N$ of $x=0$ and $\mu_0 > 0$ such that if
$\mu \in (-\mu_0,\mu_0) \setminus \{0\}$ then $f$ has two fixed points in $N$.
Moreover, there exists a neighbourhood $M$ of $y=0$,
a $C^{r-1}$ function $\nu(\mu)$,
and a $C^{r-3}$ function $a(\mu)$ with
\begin{equation}
\nu(0) = 0, \quad
\nu'(0) = f_{x\mu}(0,0), \quad
a(0) = \frac{2 f_{xxx}}{3 f_{xx}^2} \Big|_{(0,0)},
\label{eq:tcFormulae}
\end{equation}
such that with $\tilde{g}(y,\mu) = g(y,\nu(\mu),a(\mu))$,
\begin{enumerate}
\item
if $\mu \in (-\mu_0,\mu_0) \setminus \{0\}$
then $f|_N$ and $\tilde{g}|_M$ are
$C^{r-1}$-conjugate on the basins of attraction/repulsion of their corresponding fixed points, and
\item
if $\mu = 0$ then $f|_N$ is $C^1$-conjugate to $\tilde g|_M$.
\end{enumerate}
\label{thm:tcb}
\end{theorem}

Next we consider the pitchfork bifurcation.
This bifurcation involves one more fixed point than the last two cases
so we need an extra parameter in our extended normal form
in order to be able to match the derivatives of all the fixed points.
Specifically we use the form
\begin{equation}
y_{n+1} = g(y_n,\nu,a,b) = y_n + \nu y_n - y_n^3 + a y_n^5 + b \nu y_n^2 \,.
\label{eq:pffnf}
\end{equation}
In this form the bifurcation is supercritical in that the pair of fixed points that
bifurcate from zero are stable.
Unlike the previous cases, the subcritical bifurcation, in which the non-trivial fixed points are unstable,
cannot be obtained from transformations of the supercritical case. The methods are 
of course analogous and we will not go through the argument again here.
 
\begin{theorem}(Supercritical pitchfork bifurcation)
Suppose $f$ is $C^r$, $r \ge 7$, satisfying \eqref{eq:originAlwaysFixed} and
\begin{equation}
f_x(0,0) = 1, \quad f_{xx}(0,0) = 0, \quad f_{x\mu}(0,0) > 0, \quad \textrm{and} \quad f_{xxx}(0,0) < 0.
\label{eq:pff}
\end{equation}
Let $g$ be the truncated normal form (\ref{eq:pffnf}). There exists a neighbourhood $N$ of $x=0$ and $\mu_0 > 0$
such that if $\mu \in (-\mu_0, 0)$ then $f$ has one fixed point in $N$ ($x=0$ which is stable),
and if $\mu \in (0,\mu_0)$ then $f$ has three fixed points in $N$
($x=0$ which is unstable and two stable fixed points).
Moreover, there exists a neighbourhood $M$ of $y = 0$,
a $C^1$ function $\nu(\mu)$,
and continuous functions $a(\mu)$ and $b(\mu)$ with
\begin{equation}
\begin{array}{c}
\nu(0) = 0, \quad \nu'(0) = f_{x\mu}(0,0), \\
a(0) = \left( \frac{3 f_{xxxxx}}{10 f_{xxx}^2} - \frac{3 f_{xxxx}^2}{8 f_{xxx}^3} \middle) \right|_{(0,0)}, \quad
b(0 ) = \sqrt{\frac{6}{-f_{xxx}}} \left( \frac{f_{xxxx}}{4 f_{xxx}} + \frac{f_{xx\mu}}{2 f_{x\mu}} \middle) \right|_{(0,0)},
\end{array}
\label{eq:pfFormulae}
\end{equation}
such that with $\tilde{g}(y,\mu) = g(y,\nu(\mu),a(\mu),b(\mu))$,
\begin{enumerate}
\item
if $\mu \in (-\mu_0,\mu_0) \setminus \{0\}$
then $f|_N$ and $\tilde{g}|_M$ are
$C^{r-1}$-conjugate on the basins of attraction/repulsion of their corresponding fixed points, and
\item
if $\mu = 0$ then $f|_N$ is $C^1$-conjugate to $\tilde g|_M$.
\end{enumerate}
\label{thm:pfb}
\end{theorem}

Finally we treat period-doubling for which
\begin{equation}
y_{n+1} = g(y_n,\nu,a) = -y_n - \nu y_n + y_n^3 + a y_n^5 \,.
\label{eq:pdfnf}
\end{equation}

\begin{theorem}(Supercritical period-doubling bifurcation) 
Suppose $f$ is $C^r$, $r \ge 7$, satisfying \eqref{eq:originAlwaysFixed} and
\begin{equation}
f_x(0,0) = -1, \quad \left( 3 f_{xx}^2 + 2 f_{xxx} \middle) \right|_{(0,0)} > 0, \quad \textrm{and} \quad f_{x\mu}(0,0) < 0.
\label{eq:pdf}
\end{equation}
Let $g$ be the truncated normal form (\ref{eq:pdfnf}). There exists a neighbourhood $N$ of $x=0$ and $\mu_0 > 0$ such that
for all $\mu \in (-\mu_0,\mu_0)$ $f$ has one fixed point in $N$
($x=0$ which is stable for $\mu < 0$ and unstable for $\mu > 0$),
and if $\mu \in (0,\mu_0)$ then $f$ has a stable period-$2$ solution in $N$.
Moreover, there exists a neighbourhood $M$ of $y=0$,
a $C^1$ function $\nu(\mu)$,
and a continuous function $a(\mu)$ with
\begin{equation}
\begin{split}
\nu(0) &= 0, \quad \nu^\prime (0) = -f_{x\mu}(0,0), \\
a(0) &= \frac{1}{\left( 3 f_{xx}^2 + 2 f_{xxx} \right)^2}
\left( \frac{45}{4} f_{xx}^4 + \frac{39}{2} f_{xx}^2 f_{xxx} + 9 f_{xx} f_{xxxx} + \frac{6}{5} f_{xxxxx} \middle) \right|_{(0,0)},
\end{split}
\label{eq:pdFormulae}
\end{equation}
such that with $\tilde{g}(y,\mu) = g(y,\nu(\mu),a(\mu))$,
\begin{enumerate}
\item
if $\mu \in (-\mu_0,0)$
then $f|_N$ and $\tilde{g}|_M$ are $C^{r-1}$-conjugate on $N$,
\item
if $\mu = 0$ then $f|_N$ is $C^1$-conjugate to $\tilde g|_M$, and
\item
if $\mu \in (0,\mu_0)$
then $f|_N$ and $\tilde{g}|_M$ are
$C^{r-1}$-conjugate on the basins of repulsion of their fixed points, and
$C^{r-1}$-conjugate on the basins of attraction of their stable period-$2$ solutions.
\end{enumerate}
\label{thm:pdb}
\end{theorem}

In the case of the period-$2$ solution the basin of attraction is all points in $N\backslash\{0\}$.

\section{The transcritical bifurcation}
\label{sect:transcritical}
\setcounter{equation}{0}

The transcritical bifurcation provides an instructive first example of our approach.
Following most textbooks we treat the case in which the bifurcation occurs at $\mu = 0$
and $x=0$ is constrained to be a fixed point for all values of $\mu$ in a neighbourhood of $0$.
To make the main idea of our approach clear, we will not specify neighbourhoods and other details for the local results to be true (again, this follows the standard textbook approach). In the next section when the saddle-node bifurcation is described it will be necessary to be careful about how local neighbourhoods are defined.

Suppose $f$ is $C^r$ ($r \ge 4$) satisfying \eqref{eq:originAlwaysFixed}.
A transcritical bifurcation occurs at $(x,\mu) = (0,0)$ assuming
\begin{equation}
f_x(0,0)=1, \quad f_{x\mu}(0,0) \ne 0, \quad \textrm{and} \quad f_{xx}(0,0) \ne 0.
\end{equation}
By the transformations $x \to -x$ and/or $\mu \to -\mu$ if necessary we may assume
\begin{equation}\label{eq:tcspecialpg}
f_{xx}(0,0) < 0, \quad \textrm{and} \quad f_{x\mu}(0,0) > 0.
\end{equation}
The standard normal form (or truncated normal form) for this bifucation up to reversal of the $x$-direction is
$x \mapsto x+x(\nu -x)$, see Table~\ref{table:1}.
We will add a cubic term and look for functions $\nu(\mu)$ and $a(\mu)$ for which \eqref{eq:tcfnf}
is not just a convenient form, but a map that is (locally) differentiably conjugate to $f$.
We cannot expect to obtain a differentiable conjugacy on the whole of the neighbourhood we are working on,
unless $\mu = 0$, because there are two fixed points; but there will be two open sets whose union is the whole
neighbourhood and such that on each of these sets the dynamics is differentiably conjugate to the corresponding set.
This formalizes the idea that it is not just the location of fixed points,
but the details of the dynamics which is captured by the normal form,
and that this is the result of a differentiable change of coordinates.
The functions $\nu(\mu)$ and $a(\mu)$ are solutions of explicit equations
and the proof of their existence follows from the Implicit Function Theorem,
so no new technical apparatus needs to be introduced for this part of the analysis.

\subsection{Step 1: the standard skeleton}\label{subs:tcstep1pg}By definition $x=0$ is a fixed point for all $\mu$. A standard application of the Implicit Function Theorem shows that there is a unique $C^{r-1}$ curve of non-trivial fixed points $x(\mu)$ (see e.g.~\cite{Devaney, Wiggins}). By a routine calculation (see Appendices \ref{app:trans} and \ref{app:imp} for details),
\begin{equation}\label{eq:Hsolspg}
x(\mu)=-\frac{2f_{x\mu}}{f_{xx}} \bigg|_{(0,0)} \mu
- \frac{1}{3f_{xx}^3}\left(4f_{xxx}f_{x\mu}^2-6f_{xx\mu}f_{x\mu}f_{xx}+3f_{x\mu}f_{xx}^2\right) \Big|_{(0,0)} \mu^2
+ O(\mu^3).
\end{equation}
The derivative of the map evaluated at the origin
(i.e.~the multiplier of the fixed point $x=0$) is
\begin{equation}\label{eq:d0mu}
D_0(\mu)=1+f_{x\mu}(0,0) \mu + \frac{1}{2}f_{x\mu\mu}(0,0) \mu^2+O(\mu^3).
\end{equation}
Similarly the multiplier of $x(\mu)$ is
\begin{equation}\label{eq:D1mupg}
D_1(\mu)= 1-f_{x\mu}(0,0) \mu
+ \left(\frac{2f_{xxx}f_{x\mu}^2}{3f_{xx}^2}-\frac{1}{2}f_{x\mu\mu}\middle) \right|_{(0,0)} \mu^2+O(\mu^3),
\end{equation}  
and both $D_0$ and $D_1$ are $C^{r-1}$ functions.

\subsection{Step 2: fixed points of the normal form}
The normal form \eqref{eq:tcfnf} has 
\[
g_{x\nu}(0,0) =1, \quad g_{xx}(0,0) =-2, \quad g_{xxx}(0,0) =12a,
\]
and all other derivatives are zero.
Thus expressions for the fixed points of $g$ and their multipliers
can be read off from the previous subsection.
The multiplier of $y=0$ is
\begin{equation}\label{eq:d0tcpg}
d_0(\nu,a) = 1+\nu,
\end{equation}
with no error terms.
The non-trivial fixed point is
\[
y(\nu)= \nu +a\nu^2+O(\nu^3),
\]
with multiplier
\begin{equation}\label{eq:d1nupg}
d_1(\nu,a) = 1+\nu-2y+3ay^2.
\end{equation}

\subsection{Step 3: multiplier equivalence}
A necessary condition for the existence of a differentiable conjugacy between two maps is that the multipliers at corresponding fixed points are equal, i.e.
\[
d_0(\nu,a)=D_0(\mu), \quad d_1(\nu,a)=D_1(\mu).
\]
We can view this as a pair of equations to be solved for $\nu$ and $a$ in terms of $\mu$.
The first of these equations yields an immediate relationship between $\nu$ and $\mu$:
\begin{equation}\label{eq:nuinmupg}
\nu(\mu) = f_x(0,\mu)-1=f_{x\mu}(0,0)\mu + \frac{1}{2}f_{x\mu\mu}(0,0)\mu^2+O(\mu^3).
\end{equation}
The second equation also needs to be satisfied, and this is where the new normal form parameter $a$ comes in.
Let
\[
D(a,\mu)=d_1(\nu(\mu),a)-D_1(\mu )=\left(af^2_{x\mu}-\frac{2f_{xxx}f_{x\mu}^2}{3f_{xx}^2}\middle)
\right|_{(0,0)} \mu^2 +O(\mu^3),
\]
so multipliers at the non-trivial fixed points are equal if $D(a,\mu)=0$.
The Implicit Function Theorem cannot be applied to this equation directly
so we use the standard trick and consider the zeros of 
\[
G(a,\mu )=\begin{cases}\frac{D(a,\mu )}{\mu^2} & \mu \ne 0\\
\frac{1}{2}\frac{\partial^2D}{\partial \mu^2}(a,0)& \mu =0.
\end{cases}
\]
This is now is well-defined locally and $C^{r-3}$ with
\[
\frac{\partial G}{\partial \mu}(a, 0)= \frac{1}{6}\frac{\partial^3 D}{\partial \mu^3}(a,0) . 
\]
The function $G$ is amenable to a standard application of the Implicit Function Theorem
provided we choose $a(0)=a_0$ such that $G(a_0,0)=0$ and ensure $G$ is at least $C^1$, i.e. $r \ge 4$.
This implies
\begin{equation}\label{eq:tca0pg}
a_0=\frac{2f^2_{xxx}}{3f_{xx}^2} \bigg|_{(0,0)},
\end{equation}  
and since
\[
\frac{\partial G}{\partial a}(a_0,0)=f_{x\mu}(0,0)^2\ne 0,
\]
the Implicit Function Theorem guarantees a unique branch of $C^{r-3}$
solutions transversely through $(a_0,0)$.

\subsection{Step 4: differentiable conjugacies}
If $\mu =0$ then the map $f$ is $x \mapsto x+\frac{1}{2}f_{xx}x^2+O(x^3)$ and so it is differentiably conjugate to
$x\mapsto x+x^2$ (the normal form \eqref{eq:tcfnf} with $\nu =0$) by Theorem~\ref{thm:takens} \cite{KCG1990,Ofarrell2011,Takens1973,Young}. This is the source of the common view of normal form theory: the non-hyperbolic cases, having only one fixed point, are smoothly conjugate to the normal form on an open neighbourhood of the fixed point; there is no need for conjugacies on different subsets of the neighbourhood.

Suppose all the conditions described hold. If $\mu >0$ then there is a neighbourhood $N$ of $x=0$ such that if $\mu\in(-\mu_0,\mu_0)\backslash\{ 0\}$ there are two fixed points $x_1<x_2$ in $N$.  Similarly there a neighbourhood $W$ of $y=0$ such that the modified normal form has corresponding fixed points $y_1<y_2$ in $W$ if $\nu\ne 0$. Moreover there are $C^{r-3}$  functions $\nu(\mu)$ and $a(\mu)$ such that by Belitskii's Theorem (Theorem~\ref{thm:difffp}) for each $\mu$ locally there is a $C^{r-1}$-conjugacy between the two maps at corresponding values of $\mu$ and $\nu$ on $N\cap \{x<x_2\}$ and $N\cap\{x>x_1\}$; these sets are the basins of attraction and repulsion of the fixed points in $N$. 

\medskip
\emph{Remark:} Although the differentiable conjugacy is on two intervals if $\mu\ne 0$, rather than one for the non-hyperbolic case $\mu =0$, it can be smoother on its domain of definition. There are obstructions to making the local differentiable conjugacy $C^2$ in the non-hyperbolic case which are not present in the hyperbolic cases \cite{Young}.

It could be argued that the expressions here are complicated by the fact that we chose to have a coefficient of unity on the $x^2$-term of the normal form. This does indeed complicate the calculation of the connection between the two maps, but retains the simplicity of the normal form as far as possible and we prefer not to add further constants into the normal form.

\section{The saddle-node bifurcation}
\label{sect:saddlenode}
\setcounter{equation}{0}

In this section we provide an alternative approach to the saddle-node bifurcation theorem.
This relies on the standard approach to the existence of fixed points using the Implicit Function Theorem (IFT)
but then uses the results of \S\ref{sect:basics} to create a differentiable conjugacy
on different neighbourhoods of the bifurcation point.

Suppose
\begin{equation}\label{eq:nsn}
x_{n+1}=f(x_n,\mu), \quad f(0,0)=0, \quad f_x(0,0)=1,
\end{equation}
where $f$ is $C^r$ with $r\ge 4$.
If $f_\mu (0,0)$ and $f_{xx}(0,0)$ are both non-zero, the standard genericity conditions for the saddle-node bifurcation, then by a change of the sign of $x$ and $\mu$ where necessary we may assume that
\begin{equation}\label{eq:sngen}
f_{xx}(0,0)<0, \quad f_\mu (0,0)>0. 
\end{equation}

\subsection{Step 1: the standard skeleton}

We start with a standard result regarding the fixed points of $f$,
though usually the Implicit Function Theorem is used to obtain $\mu(x)$ instead of $x(\mu)$.
See Appendix \ref{app:sn} for the full calculation. 

\begin{lemma}Consider (\ref{eq:nsn}) satisfying (\ref{eq:sngen}). There is a neighbourhood $N_0$ of $(0,0)$ such that for $(x,\mu)\in N_0$ the only fixed points of (\ref{eq:nsn}) are two branches in $\mu \ge 0$ which can be written as functions of $m\ge 0$, $m^2=\mu $, with, for $k\in\{1,2\}$,
\begin{equation}\label{eq:F12}
x_{k}(m)=(-1)^k \sqrt{\frac{-2f_\mu}{f_{xx}}} \Bigg|_{(0,0)}m
+\left(\frac{f_\mu f_{xxx}-3f_{x\mu}f_{xx}}{3f_{xx}^2} \middle) \right|_{(0,0)} m^2+O(m^3).
\end{equation}
Moreover, their multipliers are
\begin{equation}\label{eq:Fderiv12}
f_x(x_{k}(m),m^2) =1+(-1)^{k+1}\sqrt{-2f_\mu f_{xx}} \,\big|_{(0,0)} m
- \frac{2}{3}\frac{f_\mu f_{xxx}}{f_{xx}} \bigg|_{(0,0)} m^2+O(m^3).
\end{equation}
\label{le:snFixedPoints}
\end{lemma}

\subsection{Step 2: fixed points of the normal form}The normal form \eqref{eq:snfnf} has
\[
g(0,0)=0, \quad g_y(0,0)=1, \quad g_\nu(0,0) =1, \quad g_{yy}(0,0) = -2, \quad g_{yyy}(0,0) =6a.
\]
So applying the results of the previous section we have for $\nu =n^2$, $n\ge 0$, fixed points 
\[ y_k(n)=(-1)^k n +\frac{1}{2}an^2 +O(n^3),
\]
for $k\in\{1,2\}$.
Moreover from \eqref{eq:Fderiv12} their multipliers are
\begin{equation}\label{eq:nfFxsn}
g_y(y_{k}(n),n^2) = 1 + (-1)^{k+1} 2n+2an^2+O(n^3).
\end{equation}

\subsection{Step 3: multiplier equivalence}
The two equations that equate the multipliers of the corresponding fixed points of $f$ and $g$ are
\[
f_x(x_k(m),m^2) = g_y(y_{k}(n),n^2), \quad k=1,2.
\]
We wish to solve these equations to obtain $n=n(m)$ and $a=a(m)$. However, the Implicit Function Theorem
cannot be applied directly to these equations as the first derivatives vanish, so we use a standard trick
(e.g.~\cite{Iooss}) and set $n(m) = m p(m)$.
Let  
\[
K_k(a,p,m) = g_y \!\left( y_{k}(p m),(p m)^2 \right) - f_x(x_k(m),m^2), \quad k=1,2.
\]
Then our problem is to solve these equations for solutions $a(m)$ and $p(m)$.
Yet the Implicit Function Theorem cannot be applied directly
to the equations $K_1=0$ and $K_2=0$, as these have a similar structure,
so we instead consider the combination $K_1=0$ and $K_1+K_2=0$.
Set 
\[
P(a,p,m)=\begin{cases}\frac{K_1(a,p,m)}{m} & \textrm{if}~m\ne 0,\\ \frac{\partial K_1}{\partial m}(a,p,m) & \textrm{if}~m=0.
\end{cases}
\]
Then from (\ref{eq:Fderiv12}) and (\ref{eq:nfFxsn})
\[
P(a,p,m)=2p-\sqrt{-2f_\mu f_{xx}} \,\big|_{(0,0)}
+\left( 2ap^2+\frac{2f_\mu f_{xxx}}{3f_{xx}} \middle) \right|_{(0,0)} m+O(m^2).
\]
Thus $P(a,p_0,0)=0$ if $p(0)=p_0$ with
\begin{equation}\label{eq:p0def}
p_0=\frac{1}{2}\sqrt{-2f_\mu f_{xx}} \,\big|_{(0,0)} \ne 0.
\end{equation}
Similarly write
\[
4Q(a,p,m)=\begin{cases}\frac{K_1+K_2}{m^2} &\textrm{if}~m\ne 0,\\
\frac{1}{2}\frac{\partial^2}{\partial m^2}(K_1+K_2) & \textrm{if}~m=0.
\end{cases}
\]
From (\ref{eq:Fderiv12}) and (\ref{eq:nfFxsn}),
\[
Q(a,p,m)=ap^2+\frac{f_\mu f_{xxx}}{3f_{xx}} \bigg|_{(0,0)} +O(m),
\]
and so $Q(a_0,p_0,0)=0$ if
\begin{equation}\label{eq:a0def}
a_0=-\frac{1}{p_0^2}\left(\frac{f_\mu f_{xxx}}{3f_{xx}} \middle) \right|_{(0,0)} =\frac{2f_{xxx}}{3f_{xx}^2} \bigg|_{(0,0)}.
\end{equation}

We now look for solutions to the pair $P(a,p,m)=0$ and $Q(a,p,m)=0$ through $(a_0,p_0)$.
This pair is at least $C^1$ because $K_1$ and $K_2$ are $C^{r-1}$,
so $P$ and $Q$ are $C^{r-2}$ and $C^{r-3}$ respectively,
and by assumption $r \ge 4$.
For the matrix of partial derivatives,
\[
\det \!\left(\begin{bmatrix}\frac{\partial P}{\partial a} & \frac{\partial P}{\partial p}\\ \frac{\partial Q}{\partial a} & \frac{\partial Q}{\partial p}\end{bmatrix}\right) = -2p_0^2\ne 0,
\] 
so the Implicit Function Theorem can indeed be applied
resulting in unique local $C^{r-3}$ solutions $a(m)$ and $p(m)$ with $a(0)=a_0$ and $p(0)=p_0$. 

\subsection{Step 4: differentiable conjugacies}
The previous subsections establish the conditions that allow Belitskii's Theorem (Theorem~\ref{thm:difffp}) to be applied directly to give the following corollary if $\mu >0$. 

\begin{corollary}There exist $X_0>0$ and $M_0>0$ such that if $0<\mu <M_0$ and $x\in [-X_0,X_0]$ then the map $x_{n+1}=f(x_n,\mu )$ with $f$ being $C^r$, $r\ge 4$, satisfying (\ref{eq:nsn}) and (\ref{eq:sngen}) has a stable fixed point $x_1(\mu )$ and an unstable fixed point $x_2(\mu )$ in $[-X_0,X_0]$. Moreover there exists a $C^{r-2}$ change of parameter $\nu (\mu )$ and a $C^{r-2}$ function $a(\mu )$ such that the map is differentiably conjugate to $y_{n+1}= \nu (\mu ) +y_n-y_n^2+a(\mu )y_n^3$ on the basin of attraction of $x_1$ in $[-X_0,X_0]$ and the basin of attraction of the corresponding fixed point of the normal form on a suitably chosen interval $[-Y_0,Y_0]$, and another differentiable conjugacy on the basin of repulsion of $x_2$.
\label{cor:sn}
\end{corollary}

Now we need to consider the case $\mu <0$. The issue here is the number of iterates the right hand end-point of the neighbourhood of $x=0$ takes to leave the neighbourhood. 

\begin{lemma}Suppose that the conditions on $f$ of Corollary~\ref{cor:sn} hold. Then there exist $X_0>0$, $N>0$ and $M_1>0$ such that $f^N(X_0,-M_1)=-X_0$ and there exists a unique increasing sequence $\mu_n\in (-M_1,0)$ such that $f^n(X_0,\mu_n)=-X_0$ with $\mu_n\to 0$ as $n\to \infty$. Let $Y_0$, $\mathcal{V}_1$ and $\nu_n$ be equivalent quantities for the normal form map (\ref{eq:snfnf}). Then there exists a $C^\infty$ function $\nu : [-M_1,0)\to [-\mathcal{V}_1,0)$ with $\nu (\mu_n)=\nu_n$ for all $n>N$ such that the normal form with $a(\mu )=a_0$ defined by (\ref{eq:a0def}) on $[-Y_0,Y_0]$ is differentiably conjugate to $f$ on $[-X_0,X_0]$.    
\end{lemma}

\begin{proof}
First note that $X_0$ and $M_1$ can be chosen so that $f_\mu (x,\mu)$ and $f_x(x,\mu )$ are both positive and $f(x,\mu )<x$ for all $(x, \mu)\in [-X_0,X_0]\times [-M_1,0)$. Thus by the mean value theorem 
\begin{equation}\label{eq:moneq}
\begin{array}{rll}
f(x,\tilde \mu_1)&<f(x,\tilde \mu_2) & \textrm{if}~\tilde\mu_1<\tilde\mu_2;\\
f(x_1,\mu )& <f(x_2,\mu ) & \textrm{if}~x_1<x_2; \ \ \textrm{and\ so}\\
f^n(x,\tilde \mu_1)& <f^n(x,\tilde \mu_2) & \textrm{if}~\tilde\mu_1<\tilde\mu_2, \ \ n\ge 1.
\end{array}
\end{equation}
These imply that there exists $N>0$ such that $f^N(X_0,-M_1)\le -X_1$. If the inequality is strict, then $f^N(X_0,0)>0$ and the intermediate value theorem imply that there exists $\mu_N>-M_1$ such that $f^N(X_0,\mu_N)=-X_0$. Redefining $M_1$ by $\mu_N=-M_1$ implies that the statement about $M_1$ is true. The existence of the values $\mu_n$, $n>N$ follow by a similar argument, using the intermediate value theorem inductively to show that  $f^{n+1}(X_0,\mu_n)<-X_0$ and $f^{n+1}(X_0,0)>0>-X_0$ implies that there exists $\mu_{n+1}>\mu_n$ such that $f^{n+1}(X_0,\mu_{n+1})=-X_0$. The uniqueness follows from the third inequality of (\ref{eq:moneq}).

Of course, the analogous statements hold equally for the saddle-node normal form map on $[-Y_0,Y_0]$, so for each $\mu\in (\mu_n,\mu_{n+1})$ the map is differentiably conjugate to the saddle-node map for any $\nu\in (\nu_n,\nu_{n+1})$ using the construction of the sketch proof of Theorem~\ref{thm:diffnofp} with initial intervals $(f(X_0),X_0)$ and $(g(Y_0),Y_0)$.

Now let $\nu:[-M_1,0)\to [-\mathcal{V}_1,0)$ be any $C^\infty$ function such that $\nu (\mu_n)=\nu_n$ and note that $\lim_{\mu \uparrow 0} \nu (\mu ) =0$. For each $\mu$ in $[-M_1,0)$ the map $f$ is differentiably conjugate to the normal form with parameters $\nu (\mu )$ and $a_0$.
\end{proof}
  
If $\mu=0$ then the map is differentiably conjugate to the normal form map by standard results for non-hyperbolic fixed points, Theorem~\ref{thm:takens}.

\section{The pitchfork bifurcation}
\label{sect:pitchfork}
\setcounter{equation}{0}

As in the case of the transcritical bifurcation
we make the simplifying assumption \eqref{eq:originAlwaysFixed}
that $x=0$ is a fixed point of $f$ for all sufficiently small values of $\mu$.
This enables us to write
\begin{equation}\label{eq:pf}
x_{n+1}=f(x_n,\mu)=x_n+x_nK(x_n,\mu)
\end{equation}
with
\begin{equation}\label{eq:defK}
K(0,0) = K_x(0,0) = 0, \quad K_\mu(0,0) > 0, \quad K_{xx}(0,0) < 0.
\end{equation}
These are equivalent to the constraints \eqref{eq:pff} on $f$
and assume a change in the sign of $\mu$ has possibly been applied.
(Note the analysis for the subcritical pitchfork bifurcation is essentially the same,
but in that case $f_{x\mu} f_{xxx}>0$.) 

\subsection{Step 1: the standard skeleton}
Fixed points of (\ref{eq:pf}) are $x=0$ and solutions to $K(x,\mu)=0$.
The function $K(x,\mu)$ satisfies the same conditions
as $f(x,\mu) - x$ in the saddle-node bifurcation setting
except has one less degree of differentiability.
So for $\mu > 0$ we set $\mu = m^2$ and in Lemma \ref{le:snFixedPoints}
replace derivatives of $f$ with those of $K$ to obtain $C^{r-2}$ non-trivial fixed points
\begin{equation}\label{eq:K12}
x_{k}(m)=(-1)^k \sqrt{\frac{-2K_\mu}{K_{xx}}} \Bigg|_{(0,0)} m
+\left(\frac{K_\mu K_{xxx}-3K_{x\mu}K_{xx}}{3K_{xx}^2}\middle) \right|_{(0,0)} m^2+O(m^3),
\end{equation}
for $k \in \{1,2\}$.
The cubic term in \eqref{eq:K12} will be needed below
and a derivation of its coefficient is outlined in Appendix \ref{app:pf}.
The multiplier of the trivial fixed point $x=0$ is
\begin{equation}\label{eq:trivmultpf}
f_x(0,\mu) = 1+K(0,\mu) =1+f_{x\mu}(0,0) \mu+\frac{1}{2}f_{x\mu\mu}(0,0) \mu^2+O(\mu^3).
\end{equation}
Similarly the multipliers of $x_k(m)$ are
\[
f_x(x_k(m),m^2) = 1+x_k(m) K_x(x_k(m),m^2),
\]
since $K(x_k(m),m^2)$ is identically zero by definition.
By using (\ref{eq:K12}) and re-expressing the multipliers in terms of the derivatives of $f$,
\begin{equation}\label{eq:xkmultpf}
f_x(x_k(m),m^2) = 1 - 2 f_{x\mu}(0,0) m^2 +(-1)^k B m^3+C m^4+O(m^5),
\end{equation}
where 
\begin{equation}\label{eq:B}
B=\sqrt{\frac{-3f_{\mu x}}{8f^3_{xxx}}} (f_{x\mu} f_{xxxx} + 2f_{xxx} f_{xx\mu}) \Big|_{(0,0)},
\end{equation}
and $C$ is a function of the derivatives of $f$ given in Appendix \ref{app:pf}.

\subsection{Step 2: fixed points of the normal form}
The corresponding normal form $g$, given by \eqref{eq:pffnf},
has trivial fixed point $y=0$ with multiplier
\begin{equation}\label{eq:y0mult}
g_y(0,\nu) = 1+\nu.
\end{equation}
For the non-trivial fixed points of $g$,
we assume $\nu > 0$ and set $\nu = n^2$, then apply the formula \eqref{eq:K12} to $g$ to obtain
\begin{equation}\label{eq:ynontpf}
y_{k}(n)=(-1)^k n+\frac{b}{2} n^2+O(n^3) .
\end{equation}
It should not come as a surprise that we need coefficents on two nonlinear terms.
This is because there are three fixed points if $\nu >0$ and so three multipliers need to be matched.
We achieve this by solving for $\nu$, $a$, and $b$ in the matching equations.
Possibly the surprising choice for our normal form is $y_n^5$ rather than say $y_n^4$.
This is because if a $y_n^4$ term is included then the equations for the coefficients at $\mu=0$
are quadratic rather than linear.
By separating out the orders of the terms in the normal form we can solve for $b_0$ at order $m^3$
and then for $a_0$ at $m^4$ without needing to solve nonlinear equations.
This does mean however that the fixed points need to be computed to third order.
As mentioned above this calculation is outlined in Appendix \ref{app:pf}.
The result for the normal form is  
\begin{equation}\label{eq:pfy12}
y_k(n) = (-1)^kn+\frac{b}{2} n^2 +(-1)^k\frac{1}{8}(4a +b^2)n^3+O(n^4),
\end{equation} 
and the multipliers are 
\begin{equation}\label{eq:nontrivmultpf}
g_y(y_k(n),n^2) 
= 1-2n^2 +(-1)^{k+1}bn^3+ \frac{1}{2} (4 a - b^2)n^4+O(n^5).
\end{equation}

\subsection{Step 3: multiplier equivalence}
Equating the multipliers \eqref{eq:trivmultpf} and \eqref{eq:y0mult}
of the trivial fixed points gives a simple relationship for $\nu$ as  function of $\mu$:
\begin{equation}
\nu = K(0,\mu) = f_{x\mu}(0,0) \mu+\frac{1}{2}f_{x\mu\mu}(0,0) \mu^2+O(\mu^3).
\label{eq:nupf}
\end{equation}
By substituting this into \eqref{eq:nontrivmultpf}
we eliminate $n$ so now the multipliers of the non-trivial fixed points of $g$ are
\begin{align}
D_k(a,b,m) &= 1 - 2 f_{x\mu}(0,0) m^2 + (-1)^{k+1} b f_{x\mu}(0,0)^{\frac{3}{2}} m^3 \nonumber \\
&\quad+\left( \frac{1}{2} (4 a - b^2) f_{x\mu}(0,0)^2 - f_{x\mu\mu}(0,0) \right) m^4 + O(m^5).
\label{eq:Dpf}
\end{align}
To equate the multipliers at corresponding fixed points we match (\ref{eq:xkmultpf}) and (\ref{eq:Dpf})
for each $k \in \{1,2\}$.
That is, we seek zeros of
\begin{align}
J_k(a,b,m) &= D_k(a,b,m) - f_x(x_k(m),m^2) \nonumber \\
&= (-1)^{k+1} \left( b f_{x\mu}(0,0)^{\frac{3}{2}} + B \right) m^3 \nonumber \\
&\quad+ \left( \frac{1}{2} (4 a - b^2) f_{x\mu}(0,0)^2 - f_{x\mu\mu}(0,0) - C \right) m^4 + O(m^5),
\label{eq:Jpf}
\end{align}
for $k \in \{1,2\}$.
To use the Implicit Function Theorem we will adopt the strategy in \S\ref{sect:saddlenode} and define 
\[
R(a,p,m)=\begin{cases}\frac{J_1}{m^3} &\textrm{if}~m\ne 0,\\
\frac{1}{3!}\frac{\partial^3J_1}{\partial m^3} & \textrm{if}~m=0,
\end{cases}
\]
and
\[
S(a,p,m)=\begin{cases}\frac{J_1+J_2}{m^4} &\textrm{if}~m\ne 0,\\
\frac{1}{4!}\frac{\partial^4}{\partial m^4}(J_1+J_2) & \textrm{if}~m=0.
\end{cases}
\]
By \eqref{eq:Jpf}, 
\begin{equation}
\begin{split}
R(a,b,m) &= b f_{x\mu}(0,0)^\frac{3}{2} - B + O(m), \\
S(a,b,m) &= (4 a - b^2) f_{x\mu}(0,0)^2 - 2 f_{x\mu\mu}(0,0) - 2 C + O(m),
\end{split}
\label{eq:RSleading}
\end{equation}
and these are $C^{r-5}$ and $C^{r-6}$ respectively because \eqref{eq:K12} is $C^{r-2}$.
We can solve $R(a_0,b_0,0) = S(a_0,b_0,0) = 0$ for $a_0$ and $b_0$
by using \eqref{eq:RSleading}
and our formulas for $B$ \eqref{eq:B} and $C$ (given in Appendix \ref{app:pf}).
The result is that $a_0$ and $b_0$ are given by the values in \eqref{eq:pfFormulae}.
At $(a_0,b_0,0)$,
\[
\det \!\left( \begin{bmatrix}
\frac{\partial R}{\partial a} &
\frac{\partial R}{\partial b} \\
\frac{\partial Q}{\partial a} &
\frac{\partial Q}{\partial b}
\end{bmatrix} \right)
= -4 f_{x\mu}(0,0)^{\frac{7}{2}} \ne 0,
\]
so by the Implicit Function Theorem there exist locally unique $C^{r-6}$ functions
$a(m)$ and $b(m)$ with $a(0)=a_0$ and $b(0)=b_0$ such that if $\mu >0$ then $R(a(m),b(m),m)=S(a(m),b(m),m)=0$.
With this the corresponding fixed points of $f$ and $g$
have the same multipliers for sufficiently small values of $\mu > 0$.

\subsection{Step 4: differentiable conjugacies}
Since the multipliers are the same,
the $\mu > 0$ part of Theorem \ref{thm:pfb} follows from Theorem \ref{thm:difffp}.
For $\mu < 0$ we only need to impose \eqref{eq:nupf} to match the multipliers of the trivial fixed points
and this case too follows from Theorem \ref{thm:difffp}.
Finally the $\mu = 0$ part of Theorem \ref{thm:pfb}
follows from Theorem~\ref{thm:takens} with $p=3$.

\section{The period-doubling bifurcation}
\label{sect:periodDoubling}
\setcounter{equation}{0}

Period-doubling bifurcations occur as a multiplier of a fixed point (possibly of an iterate of a map) passes through $-1$,
so restricted to the one-dimensional centre manifold we have a decreasing map.
The only periodic orbits of a decreasing map have periods one or two, and there is at most one fixed point.
Within an open interval that maps into itself there is precisely one fixed point. 

If $f$ is a decreasing map with a fixed point, then $f^2$ is an increasing map and so the results of \S\ref{sect:basics} can be applied to the second iterate.
Monotonicity implies that fixed points of $f^2$ which are not fixed points of $f$ come in pairs which are images of each other under $f$,
and if there are $n$ hyperbolic orbits of period two then these are ordered so that
\[
z_{-n}< \dots  < z_{-1}<  z_0 < z_1 <\dots <z_n
\]
with $f(z_{-k})=z_k$, $k=-n, \dots ,n$. As in \S\ref{sect:basics} we write $z_{-(n+1)}=-\infty$,   $z_{n+1}=\infty$, and $U_k=(z_{k-1},z_{k+1})$, $k=-n, \dots ,n$. The monotonic nature of $f$ implies that $f(U_0)=U_0$, $f(U_k)=U_{-k}$, $0<|k|<n$, and $f(U_{-(n+1)})\subseteq U_{n+1}$ and $f(U_{n+1})\subseteq U_{-(n+1)}$. Belitskii's Theorem (Theorem~\ref{thm:difffp}) implies that for any two maps $f$ and $g$ with the same structure of periodic points, multipliers of $f^2$ at the periodic points and sign of $f^2(x)-x$ on each corresponding interval $U_k$, the second iterates $f^2$ and $g^2$ are smoothly conjugate on $U_k$. The first question for the period-doubling bifurcation is therefore to determine whether a differentiable conjugacy between $f^2$ and $g^2$ implies a differentiable conjugacy between $f$ and $g$. This question has been answered in detail in \cite{Ofarrell2009} for more general maps allowing the fixed points of $f^2$ to be non-hyperbolic (and even non-isolated). 

We will start with some preliminary remarks. First note that if $f^2$ and $g^2$ are differentiably conjugate on some domains then
\[
f^2=h^{-1}\circ g^2 \circ h =  h^{-1}\circ g \circ (h\circ h^{-1})\circ g \circ h = (h^{-1}\circ g \circ h)^2
\]
in other words $f^2=w^2$ where $w=h^{-1}\circ g \circ h$ is differentiably conjugate to $g$. This function $w$ can be used to show that $g$ and $f$ are differentiably conjugate.

\begin{theorem}Suppose that $f$ and $g$ are $C^r$ strictly decreasing functions, $r\ge 1$, and $f^2$ and $g^2$ satisfy the conditions of Theorem~\ref{thm:difffp}. Let $U_k$ be the intervals $(z_{k-1},z_{k+1})$ defined above and $V_k$ the corresponding intervals for $g$. Then if $1\le k\le n$ there exists $\hat{h}:U_{-k}\cup U_k \to V_{-k}\cup V_k$ which is $C^r$ on each component such that $\hat{h} \circ f=g\circ \hat{h}$, i.e. $f$ and $g$ are $C^r$ conjugate on corresponding pairs of intervals. Moreover, $f$ and $g$ are also $C^r$ conjugate on $U_0$.\label{thm:h1}\end{theorem}

\begin{proof} Fix $k>0$. Let $h:U_{-k}\cup U_k \to V_{-k}\cup V_k$ be the $C^r$ conjugacy between $f^2$ and $g^2$ of Belitskii's Theorem (Theorem~\ref{thm:difffp}) and let $w=h^{-1}\circ g \circ h$. Then, as noted above, $f^2=w^2$. Define $h_1:U_{-k}\cup U_k \to U_{-k}\cup U_k$ by
\[ h_1(x)=x, \ \ \ x\in U_k, \quad h_1(x)=w\circ f^{-1}(x), \ \ \ x\in U_{-k}.
\]
Note that $f^{-1}:U_{-k}\to U_k$ and $w:U_k\to U_{-k}$ so $h_1(U_k)=U_k$ and $h_1(U_{-k})=U_{-k}$. Then since $f(x)\in U_{-k}$ if $x\in U_k$ and $f(x)\in U_{k}$ if $x\in U_{-k}$ ($k>0$) then 
\[  \textrm{if} \ x\in U_{-k} \ \ \textrm{then}\ \ h_1\circ f(x)=f(x) \in U_k,  \quad \textrm{and}  \]
\[\quad \textrm{if} \ x\in U_{k}\ \ \textrm{then}\ \  h_1\circ f(x)=w\circ f^{-1}\circ f(x)=w(x)\in U_{-k}.\]
Similarly (remember $k>0$)
\[
\textrm{if} \ x\in U_{k} \ \ \textrm{then}\ \ w\circ h_1(x)=w(x)\in U_{-k},  \quad \textrm{and} \]
\[
\textrm{if} \ x\in U_{-k} \ \ \textrm{then}\ \ w\circ h_1(x)=w^2\circ f^{-1}(x)=f^2\circ f^{-1}(x)=f(x)\in U_k.
\]
Hence $w\circ h_1=h_1\circ f$. Thus $f$ is $C^r$ conjugate to $w$ which is $C^r$ conjugate to $g$.

In the case of the remaining interval, $W_0$ which contains the fixed point $z_0$ of $f$ there is an added complication: the conjugating function $h_1$ is defined on $W_0\cap\{x\ge z_0\}$ and $W_0\cap\{x<z_0\}$ but the $C^r$ differentiability needs to extend to the point $z_0$ itself. This involves a further technicality described in detail in \cite{Ofarrell2009}: the functions $x$ and $w\circ f^{-1}$ to be equal up to $C^{r}$ terms at the origin, which is enough to imply that $f$ and $w$ are $C^r$ conjugate. Details of this part of the proof can be found in \cite{Ofarrell2009} and we will not go through the argument here.
\end{proof}
 
Having established the foundations of decreasing maps,
we can move on to the period-doubling bifurcation.
As above we assume $x=0$ is a fixed point for all small $\mu$ \eqref{eq:originAlwaysFixed}.
This allows us to write
\begin{equation}\label{eq:pd}
x_{n+1}=f(x_n,\mu)=-x_n + x_n P(x_n,\mu).
\end{equation}
For a period-doubling bifurcation we require
\begin{equation}\label{eq:pdgeneric}
f_x(0,0) = -1, \quad \left( 3 f_{xx}^2 + 2 f_{xxx} \middle) \right|_{(0,0)} \ne 0, \quad f_{x\mu}(0,0) \ne 0.
\end{equation}
We shall choose to work with maps $f$ giving a supercritical period-doubling bifurcation
and non-trivial fixed points existing for $\mu > 0$ which means
\begin{equation}\label{eq:pdpars}
\left( 3 f_{xx}^2 + 2 f_{xxx} \middle) \right|_{(0,0)} > 0, \quad f_{x\mu}(0,0) < 0,
\end{equation}
as in \eqref{eq:pdf}.
As with the pitchfork bifurcation, the subcritical case can be treated almost identically.

\subsection{Step 1: the standard skeleton}
The trivial fixed point $x=0$ has multiplier
\begin{equation}
f_x(0,\mu) = -1 + f_{x\mu}(0,0) \mu + \frac{1}{2} f_{x\mu\mu}(0,0) \mu^2 + O(\mu^3).
\label{eq:pdTrivialMult}
\end{equation}
The second iterate $f^2$ is increasing, at least locally.
From \eqref{eq:pd} we can write
\[
x_{n+2}=x_n+x_nG(x_n,\mu),
\]
where
\[
G(x,\mu) = -P(x,\mu) + (-1 + P(x,\mu)) P \big( -x + x P(x,\mu), \mu \big).
\]
From this formula we find that
\begin{equation}
\begin{split}
G_{xx}(0,0) &= -\frac{1}{3} \left( 3 f_{xx}^2 + 2 f_{xxx} \middle) \right|_{(0,0)} < 0, \\
G_\mu(0,0) &= -2 f_{x \mu}(0,0) > 0.
\end{split}
\label{eq:pdGderivs}
\end{equation}

Non-trivial fixed points of $f^2$ (these are period-two points of $f$) satisfy $G(x,\mu) = 0$.
In view of \eqref{eq:pdGderivs} the problem of solving for these fixed points
is the same as in the pitchfork case.
Thus $f^2$ for $\mu > 0$ has non-trivial fixed points $x_k(m)$, for $k \in \{1,2\}$ and with $\mu = m^2$,
given by \eqref{eq:K12}, but with $G$ instead of $K$.

The subsequent manipulations are now equivalent to those of \S\ref{sect:pitchfork}
although there are some interesting aspects of 
the iterated map in terms of dependence on $\mu$ rather than $m=\sqrt{\mu}$.
These are described in Appendix \ref{app:pd}. 

In particular the points $x_1(m)$ and $x_2(m)$ have the same multiplier
because they form a period-two solution for $f$.
For brevity we denote this multiplier $D(m)$.
This is a $C^{r-2}$ function of $m$ with first few terms
\begin{equation}\label{eq:dpd}
D(m) = 1 + 4 f_{x\mu}(0,0) m^2 + M m^4 + O(m^5),
\end{equation}
where a formula for $M$ is given in Appendix \ref{app:pd}.

\subsection{Step 2: periodic orbits of the normal form}
The normal form \eqref{eq:pdfnf} has trivial fixed point $y = 0$ with multiplier
\begin{equation}
g_y(0,\nu) = -1 - \nu.
\label{eq:pdTrivialMultg}
\end{equation}
The second iterate $g^2$ is close to the normal form of the (supercritical) pitchfork bifurcation.
For $\nu > 0$ it has two non-trivial fixed points $y_k(n)$, for $k \in \{1,2\}$ and with $\nu = n^2$.
By applying the general formula \eqref{eq:dpd} to the normal form
we find that the multiplier of each $y_k(n)$, write it as $d(a,n)$, is
\begin{equation}\label{eq:nontrivmultpd}
d(a,n) = 1 - 4 n^2 + 4(1-a) n^4 + O(n^5).
\end{equation}

\subsection{Step 3: multiplier equivalence}
By equating the multipliers \eqref{eq:pdTrivialMult} and \eqref{eq:pdTrivialMultg}
of the trivial fixed points of $f$ and $g$,
we obtain $\nu$ as a function of $\mu$ with
\begin{equation}\label{eq:0pdlead}
\nu = -f_{x\mu}(0,0) \mu - \frac{1}{2} f_{x\mu\mu}(0,0) \mu^2 + O(\mu^3),
\end{equation}
and recall $f_{x\mu}(0,0) < 0$ by assumption.

Next we equate the multipliers 
of the non-trivial fixed points of $f^2$ and $g^2$.
Already $n = n(m)$ via \eqref{eq:0pdlead}, so our task is to solve
$V(a,m) = d(a,n(m)) - D(m)$ for $a$ in terms of $m$.
By subtracting \eqref{eq:nontrivmultpd} from \eqref{eq:dpd} and using \eqref{eq:0pdlead},
\[
V(a,m) = \left( 4(1-a) f_{x\mu}(0,0)^2 + 2 f_{x\mu\mu}(0,0) - M \right) m^4 + O(m^5).
\]
Thus we define
\[
W(a,m)=\begin{cases}\frac{V(a,m)}{m^4} &\textrm{if}~m\ne 0,\\
\frac{1}{4!}\frac{\partial^4V}{\partial m^4} & \textrm{if}~m=0.
\end{cases}
\]
This function is $C^{r-6}$ and $W(a_0,0) = 0$ if
\begin{equation}
a_0 = 1 + \frac{2 f_{x\mu\mu}(0,0) - M}{4 f_{x\mu}(0,0)^2}.
\label{eq:pda0Again}
\end{equation}
Since $\frac{\partial W}{\partial a}(a_0,0) = -4 f_{x\mu}(0,0)^2 \ne 0$ and $r \ge 7$, the Implicit Function Theorem
implies the existence of a unique $C^{r-6}$ function $a(m)$, with $a(0) = a_0$,
such that $W(a(m),m) = 0$ for all sufficiently small values of $m$.

\subsection{Step 4: differentiable conjugacies}
We can now complete the proof of Theorem \ref{thm:pdb}.
For $\mu > 0$ we have identified functions $\nu(\mu)$ and $a(m)$ (which can be reinterpreted as a function of $\mu$)
such that, locally, the fixed points of $f^2$ and $g^2$ have the same multipliers.
Thus there is a differentiable conjugacy with the dynamics on the interval with the origin removed,
and on the interval between the points of period two by Belitskii's Theorem (Theorem~\ref{thm:difffp}) and Theorem~\ref{thm:h1}.
Moreover, the formulas \eqref{eq:pdFormulae} follow from \eqref{eq:0pdlead} and \eqref{eq:pda0Again}
where $M$ is given in Appendix \ref{app:pd}.

If $\mu < 0$ there is a local differentiable conjugacy between $f$
and $g$ by Sternberg's Theorem and the extension of Taken's Theorem to the decreasing case.

\section{Conclusion}
\label{sect:conc}
\setcounter{equation}{0}

In this paper we have shown how the reasonable expectation that truncated normal forms provide information that is more than simply topological is realised.
By introducing one or two extra terms in the most simple `normal forms' (Table 1)
we have shown that the resulting maps are typically locally differentiably conjugate to the general maps under consideration.
These additional terms and their coefficients satisfy simple equations which mean that they can be calculated explicitly (at least from a numerical point of view).
This amounts to a differentiable conjugacy on basins of attraction and repulsion,
so the different invariant regions have their own differentiable conjugacies.
Global differentiable conjugacies are unusual because of the multiple conditions
on multipliers that need to hold (see \cite{PGSG2021} for an interesting example),
so the reduction to local conjugacies is natural.

A calculation of the coefficients of the new terms we have introduced in practical problems
should give some sense of how far the map is from the standard truncated normal forms of the literature,
and hence they provide additional information about how close the bifurcation behaves to that of the standard form.
A similar analysis is possible in the continuous time case, and we will report on this separately \cite{GS2022}.

The differentiably conjugate normal forms we consider are not unique.
We have chosen the standard truncated normal forms to have coefficients which are as simple as possible;
we could have chosen coefficients that meant they were as close as possible to the Taylor series of the general system,
though this adds extra special coefficients to the normal forms.
Equally, there is an element of choice about the additional terms used. 
 
Our belief is that the calculation of these higher order coefficients and their dependency on parameters should become a natural part
of investigating important bifurcations; they give a more nuanced description of the dynamics than topological equivalence used hitherto.

\section*{Acknowledgements}
\setcounter{equation}{0}

The authors were supported by Marsden Fund contract MAU1809,
managed by Royal Society Te Ap\={a}rangi.

\appendix
\section{Sketch proof of Belitskii's Theorem}
\label{app:Bel}
\setcounter{equation}{0}


Sternberg \cite{Sternberg} proves that for every hyperbolic fixed point $x_k$ of $f$ there exists a neighbourhood $U$ of $x_k$ and a neighbourhood $V$ of $0$ such that $f$ on $U$ is differentiably conjugate to the linear map $\lambda x$ on $V$ where $\lambda =f^\prime (x_k)\ne 1$. Belitskii \cite{Belitskii1986} uses a slight generalization of the push-forward argument of the previous theorem to extend this to a diffeomorphism on the whole of $U_k = (x_{k-1},x_{k+1})$.

Fix $k$ and suppose $f(x)-x>0$ on $(x_k,x_{k+1})$, so $x_k$ is unstable and orbits are strictly increasing in $(x_k,x_{k+1})$. Suppose $a\in U\cap (x_k,x_{k+1})$, so $f^{-1}(a)\in U\cap (x_k,x_{k+1})$ but $f(a)$ may not be in $U\cap (x_k,x_{k+1})$. By the definition of $U$, $h$ is continuously differentiable at $a$ and we will extend $h$ to $(a,f(a))$ by defining 
\[
h(x)= \left( g\circ h\circ f^{-1} \right)\!(x), \quad x\in (a,f(a)).
\] 
Thus for $x\in (a,f(a))$
\begin{equation}\label{eq:longexp}
h^\prime (x)= \frac{g^\prime (h(f^{-1}(x))) h^\prime (f^{-1}(x))}{f^\prime (f^{-1}(x))},
\end{equation}
and 
\begin{equation}\label{eq:limdown} \lim_{x\downarrow a} h^\prime (x)= \frac{g^\prime (h(f^{-1}(a))h^\prime (f^{-1}(a))}{f^\prime (f^{-1}(a))}.
\end{equation}
But in $(f^{-1}(a),a)$, $h\circ f=g\circ h$ and $h$ is differentiable as it is in $U$, so
\[
 h^\prime (f(x))f^\prime (x)=g^\prime (h(x))h^\prime (x),
\]
and so to evaluate the limit of $h^\prime (a) $ from below we consider the limit of this equation with $x\to f^{-1}(a)$ from below, giving the same expression as the right hand side of (\ref{eq:limdown}). Hence $h$ is differentiable at $a$ and is differentiable by construction on $(a,f(a))$. Thus the neighbourhood on which $h$ is a diffeomorphism can be extended out to the open interval with upper limit $f(a)$. The same argument on $(f(a),f^2(a))$ shows that $h$ can be extended as a diffeomporphism to an open interval with upper bound $f^2(a)$, and then by induction to the open interval with upper bound $\lim_{i\to\infty}f^i(a)=x_{k+1}$.

The argument if $f(x)-x<0$ and on the interval $(x_{k-1},x_k)$ is analogous.  
\newline\rightline{$\square$}

Note that in general the differentiable conjugacy cannot be extended beyond $U_k$. This is because the convergence rates of iterates of $\lambda x$ (or its inverse) do not generally match the convergence rates of other fixed points of the map when these exist. Belitskii and others (see \cite{Ofarrell2009}) have developed invariants which determine whether the conjugacies can be extended to include more fixed points, but since these conditions are not generic we will not pursue this possibility. Theorem~\ref{thm:difffp} does not deal with the case in which there are no fixed points. 

\section{Calculations for the Transcritical Bifurcation}
\label{app:trans}
\setcounter{equation}{0}

Here we derive formulas for the fixed points of $f$ and their multipliers
in the transcritical bifurcation case.
This is done by directly manipulating power series.
In the next section we illustrate how the calculations can instead be done by implicit differentiation.

We write the map as $f(x,\mu) = x + x H(x,\mu)$ where
\begin{equation}
H(x,\mu) = c_1 x + c_2 \mu + c_3 x^2 + c_4 \mu x + c_5 \mu^2 + O \!\left( (|x|+|\mu|)^3 \right),
\label{eq:trH}
\end{equation}
with $c_1 < 0$ and $c_2 > 0$. In terms of the derivatives of $f$,
\[
c_1=\frac{1}{2}f_{xx}(0,0), \quad c_2=f_{x\mu}(0,0), \quad c_3=\frac{1}{6}f_{xxx}(0,0),
\]
and so on. Fixed points are $x=0$ and $x(\mu)$ solving $H(x(\mu),\mu) = 0$.
Since $H$ is $C^{r-1}$ and $c_1 \ne 0$, the Implicit Function Theorem guarantees a unique
local $C^{r-1}$ solution
\begin{equation}
x(\mu) = -\frac{c_2}{c_1} \,\mu + \left( -\frac{c_2^2 c_3}{c_1^3} + \frac{c_2 c_4}{c_1^{2}} - \frac{c_5}{c_1} \right) \mu^2
+ O(\mu^3),
\label{eq:trxStar}
\end{equation}
where the coefficients are obtained by matching terms in a power series.
In terms of $f$ the coefficients are, 
\[
-\frac{c_2}{c_1}=-\frac{2f_{x\mu}}{f_{xx}} \bigg|_{(0,0)}
\]
and
\[
-\frac{c_2^2 c_3}{c_1^3} + \frac{c_2 c_4}{c_1^{2}} - \frac{c_5}{c_1}
= -\frac{1}{3f_{xx}^3}\left(4f_{xxx}f_{x\mu}^2-6f_{xx\mu}f_{x\mu}f_{xx}+3f_{x\mu\mu}f_{xx}^2\right) \Big|_{(0,0)}.
\]
The derivative of the map is
\begin{align}
f_x(x,\mu) &= 1 + H(x,\mu) + x H_x(x,\mu) \nonumber \\
&= 1 + 2 c_1 x + c_2 \mu + 3 c_3 x^2 + 2 c_4 \mu x + c_5 \mu^2 + O \!\left( (|x|+|\mu|)^3 \right).
\end{align}
We evaluate this at the fixed points to obtain
\begin{align}
f_x(0,\mu) &= 1 + c_2 \mu + c_5 \mu^2 + O(\mu^3), \\
f_x(x(\mu),\mu) &= 1 - c_2 \mu + \left( \frac{c_2^2 c_3}{c_1^2} - c_5 \right) \mu^2 + O(\mu^3).
\end{align}
Once again we can write these in terms of $f$ and its derivatives:
\begin{align}
f_x(0,\mu) &= 1 + f_{x\mu}(0,0) \mu + \frac{1}{2}f_{x\mu\mu}(0,0) \mu^2 + O(\mu^3), \\
f_x(x(\mu),\mu) &= 1 - f_{x\mu}(0,0) \mu
+ \left( \frac{2f_{x\mu}^2 f_{xxx}}{3f_{xx}^2} - \frac{1}{2}f_{x\mu\mu} \middle) \right|_{(0,0)} \mu^2
+ O(\mu^3).
\end{align}

\section{Transcritical Bifurcation: implicit differentiation}
\label{app:imp}
\setcounter{equation}{0}

Since many textbooks use implict differentiation to determine coefficients of
expansions in bifurcation problems, for comparison here we repeat the first steps of the analysis
of the transcritical bifurcation using this method.

Assume there exists a solution $x(\mu)$ to $f(x,\mu) - x=0$.
By differentiating this equation with respect to $\mu$ we obtain
\[
f_xx^\prime+f_\mu -x^\prime =0.\]
At the origin $f_x(0,0)=1$ and $f_\mu(0,0)=0$ so this equation is automatically satisfied.
Differentiating again gives
\[
f_{xx}x^{\prime \,2}+2f_{x\mu}x^\prime +f_xx^{\prime\prime}+f_{\mu\mu}-x^{\prime\prime}=0,
\]
and evaluating at $\mu =0$ (giving $f_{\mu\mu}(0,0)=0$ because the origin is constrained to be a fixed point),
either
\[
x^\prime(0)=0 \quad \textrm{or} \quad x^\prime (0)=-\frac{2f_{x\mu}}{f_{xx}} \bigg|_{(0,0)}.
\]
The first possibility is the value for the trivial fixed point,
the second describes $x^\prime (0)$ for the nontrivial fixed point
and matches the coefficient of $\mu$ given in the previous section.
Finally, differentiating again gives
\[
f_{xxx}x^{\prime \,3}
+ 3f_{xx\mu}x^{\prime \,2} +3f_{xx}x^\prime x^{\prime\prime}+3f_{x\mu\mu}x^\prime +3f_{x\mu}x^{\prime\prime}+f_{\mu\mu\mu}+f_xx^{\prime\prime\prime}-x^{\prime\prime\prime}=0.
\]
By substituting the expression for $x^\prime (0)$ into this equation and noting that $f_{\mu\mu\mu}(0,0)=0$
we recover the second coefficient given above.

Calculations of the multipliers can be achieved in the same way. 

\section{Calculations for the Saddle-node Bifurcation}
\label{app:sn}
\setcounter{equation}{0}

We write the map as $f(x,\mu) = x + H(x,\mu)$ where
\begin{equation}
H(x,\mu) = c_2 \mu + c_3 x^2 + c_4 \mu x + c_5 \mu^2 + c_6 x^3 + \cdots, 
\label{eq:snH}
\end{equation}
with $c_2 = f_\mu(0,0) > 0$ and $c_3=\frac{1}{2}f_{xx}(0,0) < 0$.
To find fixed points of $f$ we solve $H(x,\mu) = 0$.
Since $c_2 \ne 0$ the Implicit Function Theorem could immediately be
used to solve $H=0$ for $\mu$.
But we wish to solve for $x$, and this requires a little more work.

We first assume $\mu \ge 0$ and write $\mu = m^2$.
Then write $x = m z$ and define
\begin{equation}
G(z,m) = \begin{cases}
\frac{H(m z,m^2)}{m^2}, & m \ne 0, \\
\frac{1}{2} \frac{\partial^2}{\partial m^2} H(m z, m^2), & m = 0.
\end{cases}
\nonumber
\end{equation}
Since $H$ is $C^r$, this function is $C^{r-2}$ and using \eqref{eq:snH} we obtain
\begin{equation}
G(z,m) = c_2 + c_3 z^2 + c_4 z m + c_6 z^3 m + O(m^2).
\label{eq:snG}
\end{equation}
Since $c_3 \ne 0$ the Implicit Function Theorem can be applied to obtain $z = z(m)$ solving $G(z,m) = 0$
provided we choose $z(0) = z_0$ such that $G(z_0,0) = 0$.
There are two choices, $z_0 = \pm \sqrt{\frac{-c_2}{c_3}}$, and with either $z(m)$ is $C^{r-2}$.

By multiplying each $z(m)$ by $m$ we obtain the desired fixed points, call them $x_1(m)$ and $x_2(m)$.
These are $C^{r-1}$ and by using \eqref{eq:snG} and matching terms of power series we readily arrive at
\begin{equation}
x_k(m) = (-1)^k \sqrt{\frac{-c_2}{c_3}} \,m + \frac{c_2 c_6 - c_3 c_4}{2 c_3^2} \,m^2 + O(m^3),
\label{eq:snfp}
\end{equation}
for $k \in \{1,2\}$.
By substituting $c_4 = f_{\mu x}(0,0)$ and $c_6 = \frac{1}{6} f_{xxx}(0,0)$ (also $c_2$ and $c_3$ given above)
we obtain the expression for $x_k(m)$ given in the main article.

By \eqref{eq:snH}, the derivative of the map is
\begin{align}
f_x(x,m^2)
= 1 + 2 c_3 x + c_4 m^2 + 3 c_6 x^2 + O \!\left( (|x|+|m|)^3 \right),
\end{align}
and by substituting \eqref{eq:snfp}
\begin{equation}\label{eq:start}
f_x(x_k(m),m^2) = 1 + 2 (-1)^{k+1} \sqrt{-c_2 c_3} \,m - \frac{2 c_2 c_6}{c_3} \,m^2 + O(m^3).
\end{equation}

\section{Calculations for the Pitchfork Bifurcation}
\label{app:pf}
\setcounter{equation}{0}

The map is $f(x,\mu) = x + x K(x,\mu)$ and we write
\begin{equation}
K(x,\mu) = c_2 \mu + c_3 x^2 + c_4 \mu x + c_5 \mu^2 + c_6 x^3 + c_7 \mu x^2 + c_{8} x^4 + \cdots,
\label{eq:pfK}
\end{equation}
with $c_2=f_{x\mu}(0,0) > 0$ and $c_3 =\frac{1}{6}f_{xxx}(0,0) < 0$.
In (\ref{eq:pfK}) we have included only the terms that are fourth order or lower in $x$ and $m$, where $\mu = m^2$.

The trivial fixed point is $x = 0$;
the non-trivial fixed points, valid for small $\mu > 0$, are $x_k(m)$, $k \in \{1,2\}$, satisfying $K(x_k(m),m^2) = 0$.
This last equation is identical to that in the saddle-node case,
so by importing \eqref{eq:snfp} we have
\begin{equation}
x_k(m) = (-1)^k \sqrt{\frac{-c_2}{c_3}} \,m + \frac{c_2 c_6 - c_3 c_4}{2 c_3^2} \,m^2 + (-1)^k L m^3 + O(m^4),
\label{eq:pffp}
\end{equation}
for some $L \in \mathbb{R}$.
For the pitchfork case we unfortunately need a formula for $L$,
so substitute \eqref{eq:pffp} into $K(x_k(m),m^2)$ to obtain
\begin{align}
K(x_k(m),m^2) &= \Bigg[ 2 (-1)^k c_3 \sqrt{\frac{-c_2}{c_3}} L + \frac{(c_2 c_6 - c_3 c_4)^2}{4 c_3^3}
+ \frac{c_2 c_4 c_6 - c_3 c_4^2}{2 c_3^2} \nonumber \\
&\quad+ c_5
- \frac{3 c_2 c_6 (c_2 c_6 - c_3 c_4)}{2 c_3^3} - \frac{c_2 c_7}{c_3} + \frac{c_2^2 c_{8}}{c_3^2} \Bigg] m^4 + O(m^5).
\end{align}
We then set the $m^4$ coefficient to zero to obtain
\begin{equation}
L = \frac{1}{2 \sqrt{-c_2 c_3}} \left[
-\frac{5 c_2^2 c_6^2}{4 c_3^2} + \frac{3 c_2 c_4 c_6}{2 c_3^2} - \frac{c_4^2}{4 c_3} + c_5 - \frac{c_2 c_7}{c_3} + \frac{c_2^2 c_{8}}{c_3^2} \right].
\label{eq:pfell}
\nonumber
\end{equation}

We now evaluate the multipliers of the fixed points.
The derivative of the map is
\begin{align}
f_x(x,m^2) &= 1 + K(x,m^2) + x K_x(x,m^2) \nonumber \\
&= 1 + c_2 m^2 + 3 c_3 x^2 + 2 c_4 m^2 x + c_5 m^4 + 4 c_6 x^3 + 3 c_7 m^2 x^2 + 5 c_{8} x^4 \nonumber \\
&\quad+ O \!\left( (|x|+|m|)^5 \right). \nonumber
\end{align}
We evaluate this at the fixed points to obtain
(after simplification involving substituting in the above formula for $L$):
\begin{align}
f_x(0,m^2) &= 1 + c_2 m^2 + c_5 m^4 + O(m^5), \\
f_x(x_k(m),m^2) &= 1 - 2 c_2 m^2 +(-1)^k B m^3 + C m^4 + O(m^5),
\label{eq:pfmult}
\end{align}
where
\begin{align}
B &= \sqrt{\frac{-c_2}{c_3^3}} (c_2 c_6 + c_3 c_4), \label{eq:eqB} \\
C &= -\frac{3 c_2^2 c_6^2}{2 c_3^3} + \frac{c_2 c_4 c_6}{c_3^2} + \frac{c_4^2}{2 c_3} - 2 c_5 + \frac{2 c_2^2 c_{10}}{c_3^2}. \label{eq:eqC}
\end{align}
These are easily rewritten in terms of the derivatives of $f$ at $(x,\mu) = (0,0)$ by using \eqref{eq:pfK}.

\section{Calculations for the Period-doubling Bifurcation}
\label{app:pd}
\setcounter{equation}{0}

The map is $f(x,\mu) = -x + x P(x,\mu)$ and we write
\begin{equation}
P(x,\mu) = b_1 x + b_2 \mu + b_3 x^2 + b_4 \mu x + b_5 \mu^2 + b_6 x^3 + b_7 \mu x^2 + b_{8} x^4 + \cdots,
\label{eq:pdH}
\end{equation}
where, as in the pitchfork case, we have included only the terms that are fourth order or lower in $x$ and $m$, where $\mu = m^2$.
The trivial fixed point is $x = 0$ with multiplier
\begin{equation}
f_x(0,m^2) = -1 + b_2 m^2 + b_5 m^4 + O(m^5).
\label{eq:pfmult0}
\end{equation}
By direct calculations we obtain $f^2(x,\mu) = x + x G(x,\mu)$ where
\begin{equation}
G(x,\mu) = c_2 \mu + c_3 x^2 + c_4 \mu x + c_5 \mu^2 + c_6 x^3 + c_7 \mu x^2 + c_{8} x^4 + \cdots,
\label{eq:pdG}
\end{equation}
with
\begin{equation}
\begin{split}
c_2 &= -2 b_2 = -2f_{x\mu} > 0, \\
c_3 &= -2(b_1^2 + b_3) = -\frac{1}{6}(3f_{xx}^2+2f_{xxx}) < 0, \\
c_4 &= -b_1 b_2  = -\frac{1}{2}f_{xx}f_{x\mu}\,, \\
c_5 &= b_2^2 - 2 b_5 = f_{x\mu}^2+f_{x\mu\mu})\,, \\
c_6 &= b_1 (b_1^2 + b_3)= \frac{1}{24}f_{xx}(3f_{xx}^2+2f_{xxx}), \\
c_7 &= -2 b_7 - 4 b_1 b_4 + 2 b_1^2 b_2 + 4 b_2 b_3 \\ 
&=\frac{1}{12}(3f_{xx}f_{xx\mu}+6f_{xx}^2f_{x\mu}+8f_{x\mu}f_{xxx}-8f_{xxx\mu})\,, \\
c_{8} &= -2 b_{10} - 6 b_1 b_6 - b_1^2 b_3 + 3 b_3^2 \\
& =-\frac{1}{60}f_{xxxxx}-\frac{1}{8}f_{xx}f_{xxxx}+\frac{1}{12}f_{xxx}^2-\frac{1}{24}f_{xx}^2f_{xxx}\,,
\end{split}
\label{eq:pdbAll}
\end{equation}
where all derivatives of $f$ are evaluated at $(x,\mu) = (0,0)$.
The non-trivial fixed points of $f^2$ are given by \eqref{eq:pffp} with \eqref{eq:pfell}.
They have the same multiplier, call it $D(m)$, given by \eqref{eq:pfmult}.
By substituting \eqref{eq:pdbAll} into \eqref{eq:pfmult} we obtain
\begin{align}
D(m) = 1 + 4 b_2 m^2 + M m^4 + O(m^5),
\end{align}
where the $m^3$ term has vanished because $c_2 c_6 + c_3 c_4 = 0$, and
\begin{equation}
M = 4 b_5 - \frac{b_2^2}{(b_1^2 + b_3)^2} \left( b_1^4 + 5 b_1^2 b_3 - 4 b_3^2 + 12 b_1 b_6 + 4 b_{8} \right).
\label{eq:pdp}
\end{equation}

The normal form $g$ has $b_2 = -1$, $b_3 = 1$, and $b_{8} = a$.
By substituting these into the above formulas
we find that the trivial fixed point of $g$ has multiplier $-1 - \nu$
and the non-trivial fixed points of $g^2$ have multiplier $d(a,n) = 1 - 4 n^2 + 4(1-a) n^4 + O(n^5)$
(where $\nu = n^2$).
By matching the multipliers of the trivial fixed points we obtain
\begin{equation}
\nu = -b_2 \mu - b_5\mu^2 + O(\mu^3).
\end{equation}
By using this and equating the multipliers of the non-trivial fixed points we obtain
\[
0 = \left( 4(1-a) b_2^2 + 4 b_5 - M \right) m^4 + O(m^5).
\]
In order for the $m^4$ term to vanish we must have
\begin{align}
a &= 1 + \frac{4 b_5 - M}{4 b_2^2} \nonumber \\
&= \frac{1}{(b_1^2+b_3)^2}\left( \frac{5}{4} b_1^4 + \frac{13}{4} b_1^2 b_3 + 3 b_1 b_6 + b_{8} \right) + O(m) \nonumber \\
&= \frac{1}{\left( 3 f_{xx}^2 + 2 f_{xxx} \right)^2}
\left( \frac{45}{4} f_{xx}^4 + \frac{39}{2} f_{xx}^2 f_{xxx} + 9 f_{xx} f_{xxxx} + \frac{6}{5} f_{xxxxx} \middle) \right|_{(0,0)}. \nonumber
\end{align}

\end{document}